\documentclass[12pt,a4paper,fleqn]{article}
\usepackage{a4wide,amsfonts,amsmath,latexsym,amssymb,euscript,graphicx,units,mathrsfs}

\usepackage[english]{babel}
\usepackage{graphicx}
\usepackage{color}
\usepackage{amssymb}
\usepackage{amssymb}
\usepackage[T1]{fontenc}
\usepackage{latexsym}
\usepackage{xypic}
\usepackage{eufrak}
\usepackage{euscript}
\usepackage{amsfonts,amsmath}
\usepackage{verbatim}
\usepackage{fancyhdr}
\usepackage{mathrsfs}
\usepackage{units}

\usepackage{hyperref}

\newtheorem{prop}{Proposition}[section]

\newtheorem{thm}[prop]{Theorem}

\renewcommand{\geq}{\geqslant}
\def\leq{\leqslant}

\newcommand{\R}{\mathbb{R}}

\def\HH{\EuFrak H}

\def\1{{\mathbf{1}}}

\def\sk{{\mathbb{D}}}

\def\1{{\mathbf{1}}}
\def\0.5{{\frac{1}{2}}}

\newenvironment{proof}[1]{\begin{trivlist}\item {\it
\bf Proof.}\quad} {\qed\end{trivlist}}

\newcommand{\qed}{\nopagebreak\hspace*{\fill}
{\vrule width6pt height6ptdepth0pt}\par}

\title{\bf Strong asymptotic independence\\ on Wiener chaos}

\author{Ivan Nourdin\footnote{Email: {\tt inourdin@gmail.com}; IN was partially supported by the
ANR Grant ANR-10-BLAN-0121.}, David Nualart\footnote{Email: {\tt  nualart@math.ku.edu}; DN was partially supported by the NSF grant  DMS1208625.} \,\,and
Giovanni Peccati\footnote{Email: {\tt giovanni.peccati@gmail.com}; GP was partially supported by the grant F1R-MTH-PUL-12PAMP  (PAMPAS), from Luxembourg University.}}

\begin{document}
\maketitle
\begin{abstract}
Let $F_n = (F_{1,n}, ....,F_{d,n})$, $n\geq 1$, be a sequence of random vectors such that, for every $j=1,...,d$, the
random variable $F_{j,n}$ belongs to a fixed Wiener chaos of a Gaussian field. We show that, as $n\to\infty$,
the components of $F_n$ are asymptotically independent if and only if ${\rm Cov}(F_{i,n}^2,F_{j,n}^2)\to 0$ for every $i\neq j$. Our findings are based on a novel inequality for vectors of multiple Wiener-It\^o integrals, and represent a substantial refining of criteria for asymptotic independence in the sense of moments
recently established by Nourdin and Rosi\'nski \cite{NouRos}.

\end{abstract}

\section{Introduction}

\subsection{Overview}

Let $X = \{X(h) :  h\in \HH\}$ be an isonormal Gaussian process over some real separable Hilbert space $\HH$ (see Section \ref{ss:basic} and Section \ref{s:prelim} for relevant definitions), and let $F_n = (F_{1,n}, ....,F_{d,n})$, $n\geq 1$, be a sequence of random vectors such that, for every $j=1,...,d$, the
random variable $F_{j,n}$ belongs the $q_j$th Wiener chaos of $X$ (the order $q_j\geq 1$ of the chaos being independent of $n$). The following result, proved by Nourdin and Rosi\'nski in \cite[Corollary 3.6]{NouRos}, provides a useful criterion for the asymptotic independence of the components of $F_n$.
\begin{thm}[See \cite{NouRos}]\label{t:intro} Assume that, as $n\to\infty$ and for every $i\neq j$,
\begin{equation}\label{e:scott}
{\rm Cov}(F_{i,n}^2,F_{j,n}^2)\to 0\quad\mbox{and}\quad
F_{j,n}\overset{\rm law}{\to} U_j,
\end{equation}
where each $U_j$ is a {\it moment-determinate}\footnote{Recall
that a random variable $U$ with moments of all orders is said to be {\it moment-determinate} if $E[X^n] = E[U^n]$ for every $n=1,2,...$ implies that $X$ and $U$ have the same distribution.} random variable. Then,
\[
F_{n}\overset{\rm law}{\longrightarrow} (U_1,\ldots,U_d),
\]
where the $U_j$'s are assumed to be mutually stochastically independent.
\end{thm}

In  words, Theorem \ref{t:intro} allows one to deduce joint convergence from the componentwise convergence of the elements of $F_n$, provided the limit law of each sequence $\{F_{j,n}\}$ is moment-determinate and
the covariances between the squares of the distinct components of $F_n$ vanish asymptotically. This result and its generalisations have already led to some important applications, notably in connection with time-series analysis and with the asymptotic theory of homogeneous sums --- see \cite{baitaqqu,bourguinbreton, NouRos}. The aim of this paper is to study the following important question, which was left open in the reference \cite{NouRos}:

\medskip

\noindent{\bf Question A}. {\em In the statement of Theorem \ref{t:intro}, is it possible to remove the moment-determinacy assumption for the random variables $U_1,...,U_d$?}

\medskip

Question A is indeed very natural. For instance, it is a well-known fact (see \cite[Section 3]{slud}) that non-zero random variables living inside a fixed Wiener chaos of order $q\geq 3$ are not necessarily moment-determinate, so that Theorem \ref{t:intro} cannot be applied in several contexts where the limit random variables $U_j$ have a chaotic nature. Until now, such a shortcoming has remarkably restricted the applicability of Theorem \ref{t:intro} --- see for instance the discussion contained in \cite[Section 3]{baitaqqu}.

\smallskip

In what follows, we shall derive several new probabilistic estimates (stated in Section \ref{ss:main} below) for chaotic random variables, leading to a general positive answer to Question A. As opposed to the techniques applied in \cite{NouRos}, our proof does not make use of combinatorial arguments. Instead, we shall heavily rely on the use of Malliavin calculus and Meyer inequalities (see the forthcoming formula \eqref{e:meyer}). This new approach will yield several quantitative extensions of Theorem \ref{t:intro}, each having its own interest. Note that, in particular, our main results immediately confirm Conjecture 3.7 in \cite{baitaqqu}.

\smallskip

The content of the present paper represents a further contribution to a recent and very active direction of research, revolving around the application of Malliavin-type techniques for deriving probabilistic approximations and limit theorems, with special emphasis on normal approximation results (see \cite{nualartpeccati,peccatitudor} for two seminal contributions to the field, as well as \cite{nppams} for recent developments).The reader is referred to the book \cite{nourdinpeccatibook} and the survey  \cite{nourdinsurvey} for an overview of this area of research. One can also consult the constantly updated webpage
\cite{webpage} for literally hundreds of results related to the findings contained in \cite{nualartpeccati,peccatitudor} and their many ramifications.

\subsection{Some basic definitions and notation}\label{ss:basic}

We refer the reader to \cite{nourdinpeccatibook, nualartbook} for any unexplained definition or result.
\smallskip

Let $\EuFrak H$ be a real separable infinite-dimensional Hilbert space. For any integer $q\geq 1$, let $%
\EuFrak H^{\otimes q}$ be the $q$th tensor product of $\EuFrak H$. Also, we denote
by $\EuFrak H^{\odot q}$ the $q$th symmetric tensor product. From now on, the symbol $X=\{X(h) : h\in \EuFrak H\}$ will indicate an isonormal Gaussian process on
$\EuFrak H$, defined on some probability space $(\Omega ,\mathcal{F},P)$. In particular, $X$ is a centered Gaussian family with covariance given by $E[X(h)X(g)]=\langle h,g\rangle_\HH$. We will also assume that $\mathcal{F}$ is generated by $X$.

\smallskip

For every integer $q\geq 1$,  we let $\mathcal{H}_{q}$ be the $q$th {\it Wiener chaos} of $X$,
that is, $\mathcal{H}_{q}$ is the closed linear subspace of $L^{2}(\Omega)$
generated by the class $\{H_{q}(X(h)): h\in \EuFrak H,\left\|
h\right\| _{\EuFrak H}=1\}$, where $H_{q}$ is the $q$th Hermite polynomial defined by
$$
H_q(x)=\frac{(-1)^q}{q!}e^{x^2/2}\frac{d^q}{dx^q}\big(e^{-x^2/2}\big).
$$
We denote by $\mathcal{H}_{0}$ the space of constant random variables. For
any $q\geq 1$, the mapping $I_{q}(h^{\otimes q})=q!H_{q}(X(h))$ can be extended a
linear isometry between $\EuFrak H^{\odot q}$
(equipped with the modified norm $\sqrt{q!}\left\| \cdot \right\| _{\EuFrak %
H^{\otimes q}}$) and $\mathcal{H}_{q}$ (equipped with the $L^2(\Omega)$ norm). For $q=0$, by convention $%
\mathcal{H}_{0}=\mathbb{R}$, and $I_{0}$ is the identity map.

\smallskip

It is well-known (Wiener chaos expansion) that $L^{2}(\Omega)$
can be decomposed into the infinite orthogonal sum of the spaces $\mathcal{H}%
_{q}$, that is: any square-integrable random variable $F\in L^{2}(\Omega)$ admits the following chaotic expansion:
\begin{equation}
F=\sum_{q=0}^{\infty }I_{q}(f_{q}),  \label{E}
\end{equation}%
where $f_{0}=E[F]$, and the $f_{q}\in \EuFrak H^{\odot q}$, $q\geq 1$, are
uniquely determined by $F$. For every $q\geq 0$, we denote by $J_{q}$ the
orthogonal projection operator on the $q$th Wiener chaos. In particular, if $%
F\in L^{2}(\Omega)$ is as in (\ref{E}), then $%
J_{q}F=I_{q}(f_{q})$ for every $q\geq 0$.

\subsection{Main results}\label{ss:main}

The main achievement of the present paper is the explicit estimate \eqref{e:mainresult}, appearing in the
forthcoming Theorem \ref{t:main}. Note that, in order to obtain more readable formulae, we only consider multiple
integrals with unit variance: one can deduce
bounds in the general case by a standard rescaling procedure.

\medskip

\noindent{\bf Remark on notation.} {\rm Fix integers $m,q\geq 1$. Given a smooth function $\varphi : \R^m \rightarrow \R$, we shall use the notation
\[
\|\varphi\|_q := \|\varphi\|_\infty + \sum \left\| \frac{\partial^k\varphi}{\partial x_{i_1}^{k_1} \cdots \partial x_{i_p}^{k_p}}  \right\|_\infty,
\]
where the sum runs over all $p=1,...,m$, all $\{i_1,...,i_p\}\subset \{1,...,m\}$, and all multi indices $(k_1,...,k_p) \in \{1,2,...\}^p$ verifying $k_1+\cdots +k_p := k\leq q$.
}

\medskip

\begin{thm}\label{t:main} Let $d\geq 2$ and let $ q_1\geq q_2 \geq \cdots \geq q_d\geq 1$ be fixed integers. There exists a constant $c$, uniquely depending on $d$ and $(q_1,...,q_d)$, verifying the following bound for any $d$-dimensional vector
\[
F = (F_1,...,F_d),
\]
such that $F_j = I_{q_j}(f_j)$, $f_j \in \HH^{\odot q_j}$ ($j=1,...,d$) and $E[F_j^2]=1$ for $j=1,...,d-1$, and
for any collection of smooth test functions $\psi_1,...,\psi_d:\R\to\R$,
\begin{eqnarray}\label{e:mainresult}
\left| E\left[ \prod_{j=1}^d\psi_j(F_j)\right] - \prod_{j=1}^d E[\psi_j(F_j)]\right|
&\leq& c\, \| \psi'_d\|_\infty  \prod_{j=1}^{d-1}\|\psi_j\|_{q_1}  \sum_{1\leq j<\ell \leq d} {\rm Cov}(F_j^2,F_\ell^2) .
\end{eqnarray}

\end{thm}

When applied to sequences of multiple stochastic integrals, Theorem \ref{t:main} allows one to deduce the
following strong generalization of \cite[Theorem 3.4]{NouRos}.

\begin{thm}\label{t:nr+}
Let $d\geq 2$ and let $ q_1\geq q_2 \geq \cdots \geq q_d\geq 1$ be fixed integers. For every $n\geq 1$, let
$F_n = (F_{1,n},...,F_{d,n})$
be a $d$-dimensional random vector such that $F_{j,n} = I_{q_j}(f_{j,n})$, with
$f_{j,n} \in \HH^{\odot q_j}$ and $E[F_{j,n}^2]=1$ for all $1\leq j\leq d$ and $n\geq 1$.
Then, the following three conditions are equivalent, as $n\to\infty$:
\begin{enumerate}
\item[\rm (1)] ${\rm Cov}( F^2_{i,n}, F_{j,n}^2) \to 0$ for every $1\leq i\neq j\leq d$;
\item[\rm (2)] $\|f_{i,n}\otimes_r f_{j,n}\|\to 0$ for every $1\leq i\neq j\leq d$ and $1\leq r\leq q_i\wedge q_j$;
\item[\rm (3)] The random variables $F_{1,n},...,F_{d,n}$ are asymptotically independent, that is, for every collection of smooth bounded test functions $\psi_1,...,\psi_d : \R\to \R$,
\[
E\left[ \prod_{j=1}^d \psi_j(F_{j,n})\right] - \prod_{j=1}^d E[\psi_j(F_{j,n})]\longrightarrow  0.
\]
\end{enumerate}
\end{thm}

We can now state the announced extension
of Theorem \ref{t:intro} (see Section 1), in which the determinacy condition
for the limit random variables $U_j$ has been eventually removed.

\begin{thm}\label{utile}
Let $d\geq 2$ and let $ q_1\geq q_2 \geq \cdots \geq q_d\geq 1$ be fixed integers. For every $n\geq 1$, let
$F_n = (F_{1,n},...,F_{d,n})$
be a $d$-dimensional random vector such that $F_{j,n} = I_{q_j}(f_{j,n})$, with
$f_{j,n} \in \HH^{\odot q_j}$ and $E[F_{j,n}^2]=1$ for all $1\leq j\leq d$ and $n\geq 1$.
Let $U_1,\ldots,U_d$ be independent random variables such that
$F_{j,n}\overset{\rm law}{\to}U_j$ as $n\to\infty$ for every $1\leq j\leq d$.
Assume that either Condition {\rm (1)} or Condition {\rm (2)} of Theorem \ref{t:nr+} holds. Then, as $n\to\infty$,
\[
F_{n}\overset{\rm law}{\to}(U_1,\ldots,U_d).
\]
\end{thm}

By considering linear combinations, one can also prove the following straightforward generalisations of Theorem \ref{t:nr+} and Theorem \ref{utile} (which are potentially useful for applications), where each component of the vector $F_n$ is replaced by a multidimensional object. The simple proofs are left to the reader.

\begin{prop}\label{p:nr+2}
Let $d\geq 2$, let $ q_1\geq q_2 \geq \cdots \geq q_d\geq 1$ and $m_1,...,m_d\geq 1$ be fixed integers, and set $M:=\sum_{j=1}^d m_j$.  For every $j=1,...,d$, let
\[
{\bf F}_{j,n} = (F_{j,n}^{(1)},...,F_{j,n}^{(m_j)}):= \big(I_{q_j}(f^{(1)}_{j,n}),...,I_{q_j}(f^{(m_j)}_{j,n})\big),
\]
where, for $\ell = 1,...,m_j$, $f^{(\ell)}_{j,n} \in \HH^{\odot q_j}$ and $E[(F_{j,n}^{(\ell)})^2] =q_j! \| f^{(\ell)}_{j,n}\|^2_{\HH^{\otimes q_j}}=1$.
Finally, for every $n\geq 1$, write ${\bf F}_n$ to indicate the $M$-dimensional vector
$({\bf F}_{1,n},...,{\bf F}_{d,n})$.
Then, the following three conditions are equivalent, as $n\to\infty$:
\begin{enumerate}
\item[\rm (1)] ${\rm Cov}\big( (F^{(\ell)}_{i,n})^2, (F^{(\ell')}_{j,n})^2\big) \to 0$ for every $1\leq i\neq j\leq d$, every $\ell=1,...,m_i$ and every $\ell' = 1,...,m_j$;
\item[\rm (2)] $\|f^{(\ell)}_{i,n}\otimes_r f^{(\ell')}_{j,n}\|\to 0$ for every $1\leq i\neq j\leq d$, for every $1\leq r\leq q_i\wedge q_j$, every $\ell=1,...,m_i$ and every $\ell' = 1,...,m_j$;
\item[\rm (3)] The random vectors ${\bf F}_{1,n},...,{\bf F}_{d,n}$ are asymptotically independent, that is: for every collection of smooth bounded test functions $\psi_j : \R^{m_j} \to \R$, $j=1,...,d$,
\[
E\left[ \prod_{j=1}^d \psi_j({\bf F}_{j,n})\right] - \prod_{j=1}^d E[\psi_j({\bf F}_{j,n})]\longrightarrow  0.
\]
\end{enumerate}
\end{prop}

\begin{prop}\label{utile2}
Let the notation and assumptions of Proposition \ref{p:nr+2} prevail, and assume that either Condition {\rm (1)} or Condition {\rm (2)} therein is satisfied. Consider a collection $({\bf U}_1,...,{\bf U}_{d})$ of independent random vectors such that, for $j=1,...,d$, ${\bf U}_j$ has dimension $m_j$. Then, if ${\bf F}_{j,n} $ converges in distribution to ${\bf U}_j$, as $n\to\infty$, one has also that
\[
{\bf F}_n \overset{\rm law}{\to}({\bf U}_1,\ldots,{\bf U}_d).
\]
\end{prop}

\bigskip

The plan of the paper is as follows. Section 2 contains some further preliminaries related to Gaussian analysis and Malliavin calculus. The proofs of our main results are gathered in Section 3.

\section{Further notation and results from Malliavin calculus} \label{s:prelim}

Let $\{e_{k},\,k\geq 1\}$ be a complete orthonormal system in $\EuFrak H$.
Given $f\in \EuFrak H^{\odot p}$, $g\in \EuFrak H^{\odot q}$ and $%
r\in\{0,\ldots ,p\wedge q\}$, the $r$th {\it contraction} of $f$ and $g$
is the element of $\EuFrak H^{\otimes (p+q-2r)}$ defined by
\begin{equation}
f\otimes _{r}g=\sum_{i_{1},\ldots ,i_{r}=1}^{\infty }\langle
f,e_{i_{1}}\otimes \ldots \otimes e_{i_{r}}\rangle _{\EuFrak H^{\otimes
r}}\otimes \langle g,e_{i_{1}}\otimes \ldots \otimes e_{i_{r}}\rangle _{%
\EuFrak H^{\otimes r}}.  \label{v2}
\end{equation}%
Notice that $f\otimes _{r}g$ is not necessarily symmetric. We denote its
symmetrization by $f\widetilde{\otimes }_{r}g\in \EuFrak H^{\odot (p+q-2r)} $%
. Moreover, $f\otimes _{0}g=f\otimes g$ equals the tensor product of $f$ and
$g$ while, for $p=q$, $f\otimes _{q}g=\langle f,g\rangle _{\EuFrak %
H^{\otimes q}}$. In the particular case $\EuFrak H=L^{2}(A,\mathcal{A},\mu )$, where $%
(A,\mathcal{A})$ is a measurable space and $\mu $ is a $\sigma $-finite and
non-atomic measure, one has that $\EuFrak H^{\odot q}$ can be identified with the space $L_{s}^{2}(A^{q},%
\mathcal{A}^{\otimes q},\mu ^{\otimes q})$ of $\mu^q$-almost everywhere symmetric and
square-integrable functions on $A^{q}$. Moreover, for every $f\in \EuFrak %
H^{\odot q}$, $I_{q}(f)$ coincides with the multiple Wiener-It\^{o} integral
of order $q$ of $f$ with respect to $X$
and (\ref{v2}) can be written as
\begin{eqnarray*}
&&(f\otimes _{r}g)(t_1,\ldots,t_{p+q-2r})
=\int_{A^{r}}f(t_{1},\ldots ,t_{p-r},s_{1},\ldots ,s_{r}) \\
&&\hskip2cm \times \ g(t_{p-r+1},\ldots ,t_{p+q-2r},s_{1},\ldots ,s_{r})d\mu
(s_{1})\ldots d\mu (s_{r}).
\end{eqnarray*}

\medskip

We will now introduce some basic elements of the Malliavin calculus with respect
to the isonormal Gaussian process $X$ (see again \cite{nourdinpeccatibook, nualartbook} for any unexplained notion or result). Let $\mathcal{S}$
be the set of all smooth and cylindrical random variables of
the form
\begin{equation}
F=g\left( X(\phi _{1}),\ldots ,X(\phi _{n})\right) ,  \label{v3}
\end{equation}%
where $n\geq 1$, $g:\mathbb{R}^{n}\rightarrow \mathbb{R}$ is a infinitely
differentiable function with compact support, and $\phi _{i}\in \EuFrak H$.
The Malliavin derivative of $F$ with respect to $X$ is the element of $%
L^{2}(\Omega ,\EuFrak H)$ defined as
\begin{equation*}
DF\;=\;\sum_{i=1}^{n}\frac{\partial g}{\partial x_{i}}\left( X(\phi
_{1}),\ldots ,X(\phi _{n})\right) \phi _{i}.
\end{equation*}
By iteration, one can
define the $q$th derivative $D^{q}F$ for every $q\geq 2$, which is an element of $L^{2}(\Omega ,%
\EuFrak H^{\odot q})$.

\smallskip

For $q\geq 1$ and $p\geq 1$, ${\mathbb{D}}^{q,p}$ denotes the closure of $%
\mathcal{S}$ with respect to the norm $\Vert \cdot \Vert_{\mathbb{D}^{q,p}}$, defined by
the relation
\begin{equation*}
\Vert F\Vert _{\mathbb{D}^{q,p}}^{p}\;=\;E\left[ |F|^{p}\right] +\sum_{i=1}^{q}E\left[
\Vert D^{i}F\Vert _{\EuFrak H^{\otimes i}}^{p}\right].
\end{equation*}
The {\it Malliavin derivative} $D$ verifies the following chain rule. If $%
\varphi :\mathbb{R}^{n}\rightarrow \mathbb{R}$ is continuously
differentiable with bounded partial derivatives and if $F=(F_{1},\ldots
,F_{n})$ is a vector of elements of ${\mathbb{D}}^{1,2}$, then $\varphi
(F)\in {\mathbb{D}}^{1,2}$ and
\begin{equation}\label{chain}
D\varphi (F)=\sum_{i=1}^{n}\frac{\partial \varphi }{\partial x_{i}}%
(F)DF_{i}.
\end{equation}
Note also that a random variable $F$ as in (\ref{E}) is in ${\mathbb{D}}%
^{1,2}$ if and only if
\begin{equation*}
\sum_{q=1}^{\infty }qq!\Vert f_{q}\Vert _{\EuFrak H^{\otimes q}}^{2}<\infty ,
\end{equation*}%
and, in this case, $E\left[ \Vert DF\Vert _{\EuFrak H}^{2}\right]
=\sum_{q\geq 1}qq!\Vert f_{q}\Vert _{\EuFrak H^{\otimes q}}^{2}$. If $\EuFrak H=%
L^{2}(A,\mathcal{A},\mu )$ (with $\mu $ non-atomic), then the
derivative of a random variable $F$ as in (\ref{E}) can be identified with
the element of $L^{2}(A\times \Omega )$ given by
\begin{equation}
D_{a}F=\sum_{q=1}^{\infty }qI_{q-1}\left( f_{q}(\cdot ,a)\right) ,\quad a\in
A.  \label{dtf}
\end{equation}

\smallskip

We denote by $\delta $ the adjoint of the operator $D$, also called the
{\it divergence operator}.
A random element $u\in L^{2}(\Omega ,\EuFrak %
H)$ belongs to the domain of $\delta $, noted $\mathrm{Dom}\,\delta $, if and
only if it verifies
\begin{equation*}
\big|E\big[\langle DF,u\rangle _{\EuFrak H}\big]\big|\leq c_{u}\,\sqrt{E[F^2]}
\end{equation*}%
for any $F\in \mathbb{D}^{1,2}$, where $c_{u}$ is a constant depending only
on $u$. If $u\in \mathrm{Dom}\,\delta $, then the random variable $\delta (u)$
is defined by the duality relationship (customarily called `integration by parts
formula'):
\begin{equation}
E[F\delta (u)]=E\big[\langle DF,u\rangle _{\EuFrak H}\big],  \label{ipp}
\end{equation}%
which holds for every $F\in {\mathbb{D}}^{1,2}$. The formula (\ref%
{ipp}) extends to the multiple Skorohod integral $\delta ^{q}$, and we
have
\begin{equation}
E\left[ F\delta ^{q}(u)\right] =E\left[ \left\langle D^{q}F,u\right\rangle _{%
\EuFrak H^{\otimes q}}\right]  \label{dual}
\end{equation}%
for any element $u$ in the domain of $\delta ^{q}$ and any random variable $%
F\in \mathbb{D}^{q,2}$. Moreover, $\delta ^{q}(h)=I_{q}(h)$ for any $h\in %
\EuFrak H^{\odot q}$.

\smallskip

The following property, corresponding to \cite[Lemma 2.1]{NoNu}, will be used in the paper.
Let $q\geq 1$ be an integer, suppose that  $F\in {\mathbb{D}}^{q,2}$, and
let $u$ be a symmetric element in $\mathrm{Dom}\,\delta ^{q}$. Assume that
$\left\langle D^{r}F,\delta ^{j}(u)\right\rangle _{\EuFrak H^{\otimes r}}\in
L^{2}(\Omega ,\EuFrak H^{\otimes q-r-j})$ for any $\ 0\leq r+j\leq q$. Then
$\left\langle D^{r}F,u\right\rangle _{\EuFrak H^{\otimes r}}$ belongs to the
domain of $\delta ^{q-r}$ for any  $r=0,\ldots ,q-1$, and we have
\begin{equation}
F\delta ^{q}(u)=\sum_{r=0}^{q}\binom{q}{r}\delta ^{q-r}\left( \left\langle
D^{r}F,u\right\rangle _{\EuFrak H^{\otimes r}}\right) .  \label{t3}
\end{equation}
(We use the convention that $\delta^0(v)=v$, $v\in\mathbb{R}$, and $D^0F=F$, $F\in L^2(\Omega)$.)

\smallskip

For any Hilbert space $V$, we denote by $\mathbb{D}^{k,p}(V)$ the
corresponding Sobolev space of $V$-valued random variables (see \cite[page 31]{nualartbook}%
). The operator $\delta^q$   is continuous from $\mathbb{D}^{k,p}(\EuFrak H^{\otimes q})$
to $\mathbb{D}^{k-q,p}$, for any $p>1$ and  any integers $k\ge q\ge 1$, and one has the estimate

\begin{equation}
\left\| \delta^q (u)\right\| _{\mathbb{D}^{k-q,p}}\leq c_{k,p}\left\| u\right\| _{\mathbb{D}%
^{k,p}(\EuFrak H^{\otimes q})}  \label{e:meyer}
\end{equation}
for all $u\in\mathbb{D}^{k,p}(\EuFrak H^{\otimes q})$, and some constant $c_{k,p}>0$.
These inequalities are direct consequences of the so-called {\it Meyer inequalities} (see %
\cite[Proposition 1.5.7]{nualartbook}).
 In particular, these estimates imply that
$\mathbb{D}^{q,2}(\EuFrak H^{\otimes q})\subset \mathrm{Dom}\,\delta ^{q}$ for any integer $q\ge 1$.

\smallskip

The operator $L$ is defined on the Wiener chaos expansion as
\begin{equation*}
L=\sum_{q=0}^{\infty }-qJ_{q},
\end{equation*}
and is called the {\it infinitesimal generator of the Ornstein-Uhlenbeck
semigroup}. The domain of this operator in $L^{2}(\Omega)$ is the set%
\begin{equation*}
\mathrm{Dom}L=\{F\in L^{2}(\Omega ):\sum_{q=1}^{\infty }q^{2}\left\|
J_{q}F\right\| _{L^{2}(\Omega )}^{2}<\infty \}=\mathbb{D}^{2,2}\text{.}
\end{equation*}%
There is an important relationship between the operators $D$, $\delta $ and $L$. A random variable $F$ belongs to the
domain of $L$ if and only if $F\in \mathrm{Dom}\left( \delta D\right) $
(i.e. $F\in {\mathbb{D}}^{1,2}$ and $DF\in \mathrm{Dom}\,\delta $), and in
this case
\begin{equation}
\delta DF=-LF.  \label{k1}
\end{equation}
We also define the operator $L^{-1}$, which is the
{\it pseudo-inverse} of $L$, as follows: for every $F\in L^2(\Omega)$ with zero mean,
we set $L^{-1}F$ $=$ $\sum_{q\geq 1}-\frac{1}{q} J_q(F)$. We note that
$L^{-1}$ is an operator with values in $\sk^{2,2}$ and that $LL^{-1}F=E-E[F]$ for any $F\in L^2(\Omega)$.

\section{Proofs of the results stated in Section \ref{ss:main}}

\subsection{Proof of Theorem \ref{t:main}}

\medskip

\noindent The proof of Theorem \ref{t:main} is based on a recursive application of the following
quantitative result,  whose proof has been inspired by the pioneering work of  \"Ust\"unel and Zakai on the characterization of the independence on Wiener chaos (see \cite{ustunelzakai}).

\begin{prop}\label{p:main}
Let $m\geq 1$ and $p_1,...,p_m,q$ be integers such that $p_j\geq q$  for every $j=1,...,m$. There exists a constant $c$, uniquely depending on $m$ and $p_1,...,p_m, q$, such that one has the bound
\[
|E[ \varphi(F) \psi(G)]- E[ \varphi(F)]E[ \psi(G)]| \le c\|\psi'\|_\infty \|\varphi\|_q \sum_{j=1}^m {\rm Cov}(F_j^2,G^2),
\]
for every vector $F= (F_1,...,F_m)$ such that $F_j=I_{p_j}(f_j)$, $f_j\in \HH^{\odot p_j}$
and $E[F_j^2]=1$ ($j=1,...,m$), for every random variable $G=I_q(g)$, $g\in \HH^{\odot q}$,  and for every pair of smooth test functions $\varphi : \R^m\to \R$ and $\psi : \R\to \R$.
\end{prop}

\begin{proof}

Throughout the proof, the symbol $c$ will denote a positive finite constant uniquely depending on $m$ and
$p_1,...,p_m, q$, whose value may change from line to line. Using the chain rule (\ref{chain}) together with the relation $-DL^{-1}= (I-L)^{-1}D$ (see, e.g., \cite{nualartzakai}), one has
\[
 \varphi(F)- E[\varphi(F)] = LL^{-1}\varphi(F)= -\delta(DL^{-1} \varphi(F))=\sum_{j=1}^m \delta((I-L)^{-1}
 \partial_j\varphi(F)DF_j),
 \]
from which one deduces that
\begin{eqnarray*}
&&E[ \varphi(F) \psi(G)]- E[ \varphi(F)]E[ \psi(G)]  =\sum_{j=1}^mE[  \langle (I-L)^{-1} \partial_j\varphi(F) DF_j , DG \rangle_\HH \psi'(G)] \\
& \le& \|\psi'\|_\infty \sum_{j=1}^m E\big[| \langle (I-L)^{-1} \partial_j\varphi(F) DF_j , DG \rangle_\HH|\big].
\end{eqnarray*}
We shall now fix $j=1,...,m$, and consider separately every addend appearing in the previous sum.
As it is standard, without loss of generality, we can assume that the underlying Hilbert space $\HH$ is of
the form $L^2(A, \mathcal{A}, \mu)$, where $\mu$ is a $\sigma$-finite measure without atoms. It follows that
\begin{equation}  \label{eq1}
\langle (I-L)^{-1} \partial_j\varphi(F) DF_j , DG \rangle_\HH
=p_jq\int_A  [(I-L)^{-1} \partial_j\varphi(F) I_{p_j-1}(f_j(\cdot,\theta)) ] I_{q-1}(g(\cdot, \theta)) \mu(d\theta).
\end{equation}
Now we apply the formula (\ref{t3})
to $u=g(\cdot, \theta) $ and $F= (I-L)^{-1} \partial_j\varphi(F) I_{p_j-1}(f_j(\cdot,\theta))$ and we obtain, using $D^r(I-L)^{-1}=((r+1)I-L)^{-1} D^r$ as well (see, e.g., \cite{nualartzakai}),
\begin{eqnarray}
&& [(I-L)^{-1} \partial_j\varphi(F) I_{p_j-1}(f_j(\cdot,\theta)) ] I_{q-1}(g(\cdot, \theta))  \label{eq2}\\
&=&\sum_{r=0}^{q-1} {q-1\choose r }\delta^{q-1-r}\left(
\left \langle g(\cdot,\theta) ,D^r[(I-L)^{-1} \partial_j\varphi(F) I_{p_j-1}
(f_j(\cdot,\theta)) ] \right\rangle _{\HH^{\otimes r}} \right)\notag \\
& =&\sum_{r=0}^{q-1} {q-1\choose r }\delta^{q-1-r}\left(
\left \langle g(\cdot,\theta) ,((r+1)I-L)^{-1} D^r[\partial_j \varphi(F) I_{p_j-1}
(f_j(\cdot,\theta)) ] \right\rangle _{\HH^{\otimes r}}\right). \notag
\end{eqnarray}
Now, substituting (\ref{eq2}) into (\ref{eq1}) yields
\begin{eqnarray*}
&&\langle (I-L)^{-1} \partial_j\varphi(F) DF_j , DG \rangle_\HH
=p_jq\sum_{r=0}^{q-1} {q-1\choose r }\\
&& \times \delta^{q-1-r}
\left(  \int_{A^{r+1}} g(\cdot,  \mathbf{s}^{r+1} ) ((r+1)I-L)^{-1} D_{s_1, \dots, s_r}^r
[\partial_j\varphi(F) I_{p_j-1}(
f_j(\cdot, s_{r+1})) ] \mu(d\mathbf{s}^{r+1}) \right),
\end{eqnarray*}
where $ \mathbf{s}^{r+1} =(s_1, \dots, s_{r+1})$.
We have, by the Leibniz rule,
\begin{eqnarray*}
&&D_{s_1, \dots, s_r}^r[\partial_j\varphi(F) I_{p_j-1}(f_j(\cdot, s_{r+1})) ]  = \sum_{\alpha=0}^r
{r\choose \alpha}D_{s_1 ,\dots, s_\alpha} ^\alpha [\partial_j\varphi(F) ] D_{s_{\alpha+1}, \dots, s_r}^{r-\alpha}
[I_{p_j-1} (f_j(\cdot, s_{r+1})) ] \\
& =&\sum_{\alpha=0}^r  {r\choose \alpha}   \frac{(p_j-1)!}{ (p_j-r+\alpha-1)!} D_{s_1 ,\dots, s_\alpha} ^\alpha [\partial_j\varphi(F) ]
I_{p_j-r+\alpha-1} (f_j(\cdot, s_{\alpha+1} ,\dots, s_{r+1})).
\end{eqnarray*}
Fix $0\le r \le q-1$ and $0\le \alpha \le r$. It suffices to estimate the following expectation
\begin{eqnarray}
&&E \Bigg| \delta^{q-1-r} \Bigg(\int_{A^{r+1}} g(\cdot, \mathbf{s}^{\alpha} ,\mathbf{t}^{r-\alpha+1}) \label{star} \\
&& \times ((r+1)I-L)^{-1}D_{\mathbf{s}^\alpha} ^\alpha [\partial_j\varphi(F)  ]
I_{p_j-r+\alpha-1} (f_j(\cdot, \mathbf{t}^{r-\alpha+1}))\mu(d\mathbf{s}^{\alpha}) \mu(d\mathbf{t}^{r-\alpha+1}) \Bigg)
\Bigg|.\notag
\end{eqnarray}
Note that, in the previous formula, the symbol `$\cdot$' inside the argument of the kernel $g$ represents
variables that are integrated with respect to the multiple Skorohod integral $\delta^{q-1-r}$, whereas
the `$\cdot$' inside the argument of $f_j$ stands for variables that are integrated with respect to the multiple
Wiener-It\^o integral $I_{p_j-r+\alpha-1}$.
By Meyer's inequalities \eqref{e:meyer}, we can estimate the expectation (\ref{star}), up to a universal constant, by the sum over $0\leq \beta\leq q-r-1$ of the quantities
\begin{eqnarray*}
&&\Bigg(   \int_{A^{q-r-1 +\beta}} E \Bigg( \int_{A^{r+1}} g( \mathbf{v}^{q-r-1},\mathbf{s}^{\alpha},
\mathbf{t}^{r-\alpha+1} )
D^{\beta}_{\mathbf{u}^{\beta}} \big\{ ((r+1)I-L)^{-1}D_{\mathbf{s}^\alpha} ^\alpha [\partial_j\varphi(F) ]  \\
&&\times I_{p_j-r+\alpha-1}(f_j(\cdot, \mathbf{t}^{r-\alpha+1}))\big\}\mu(d\mathbf{s}^{\alpha}) \mu(d\mathbf{t}^{r-\alpha+1}) \Bigg)^2  \mu(d \mathbf{v}^{q-r-1})\mu(d\mathbf{u}^{\beta})
\Bigg)^{\frac12}  \\
&&= \Bigg(   \int_{A^{q-r-1+\beta}} E \Bigg( ((\beta+r+1)I-L)^{-1} \int_{A^{r+1}}
g( \mathbf{v}^{q-r-1},\mathbf{s}^{\alpha},
\mathbf{t}^{r-\alpha+1} )
 D^{\beta}_{\mathbf{u}^{\beta}} \big\{ D_{\mathbf{s}^\alpha} ^\alpha[\partial_j\varphi(F) ]  \\
&&\times I_{p_j-r+\alpha-1} (f_j(\cdot, \mathbf{t}^{r-\alpha+1}))\big\}\mu(d\mathbf{s}^{\alpha}) \mu(d\mathbf{t}^{r-\alpha+1})) \Bigg)^2  \mu(d \mathbf{v}^{q-r-1})\mu(d\mathbf{u}^{\beta})
\Bigg)^{\frac12}  \\
&&\le c \Bigg(   \int_{A^{q-r-1+\beta}} E \Bigg(  \int_{A^{r+1}} g( \mathbf{v}^{q-r-1},\mathbf{s}^{\alpha},
\mathbf{t}^{r-\alpha+1} )
 D^{\beta}_{\mathbf{u}^{\beta}} \big\{ D_{\mathbf{s}^\alpha} ^\alpha[\partial_j\varphi(F) ]  \\
&&\times I_{p_j-r+\alpha-1} (f_j(\cdot, \mathbf{t}^{r-\alpha+1}))\big\}\mu(d\mathbf{s}^{\alpha}) \mu(d\mathbf{t}^{r-\alpha+1})) \Bigg)^2  \mu(d \mathbf{v}^{q-r-1})\mu(d\mathbf{u}^{\beta})
\Bigg)^{\frac12} .
\end{eqnarray*}
Thanks to the Leibniz formula, the last bound implies that we need to estimate, for any $0\le \eta\le\beta \le q-r-1$, the following quantity
\begin{eqnarray*}
&&\Bigg(   \int_{A^{q-r-1+\beta}} E \Bigg(  \int_{A^{r+1}}g( \mathbf{v}^{q-r-1},\mathbf{s}^{\alpha},
\mathbf{t}^{r-\alpha+1} )
 D^{\alpha +\eta}_{ \mathbf{s}^\alpha, \mathbf{w}^\eta }[\partial_j\varphi(F) ]  \\
&&\times I_{p_j-r+\alpha-1+\eta-\beta} (f_j (\cdot, \mathbf{t}^{r-\alpha+1}, \mathbf{y}^{\beta-\eta}))\mu(d\mathbf{s}^{\alpha})\mu(d\mathbf{t}^{r-\alpha+1})  \Bigg)^2  \mu(d \mathbf{v}^{q-r-1})  \mu(d\mathbf{w}^\eta)   \mu(d\mathbf{y}^{\beta-\eta})\Bigg)^{\frac12}.
\end{eqnarray*}
We can rewrite this quantity as
\begin{eqnarray*}
&&\Bigg(\int_{A^{q-r-1-\beta}}  E   \left( \left(I_{p_j-r+\alpha-1+\eta-\beta}(f_j(\cdot, \mathbf{y}^{\beta-\eta})) \otimes_{r-\alpha+1} g(\cdot, \mathbf{v}^{q-r-1}) \right) \otimes_\alpha D^{\alpha +\eta}[\partial_j\varphi(F) ]  (\mathbf{w}^\eta) \right)^2\\
&&\times \mu(d \mathbf{y}^{\beta-\eta}) \mu( d\mathbf{w}^\eta) \mu (d \mathbf{v}^{q-r-1}) \Bigg) ^{\frac 12}.
\end{eqnarray*}
Applying the Cauchy-Schwarz inequality yields that such a quantity is bounded by
\begin{eqnarray*}
&&\left( E \left[\|I_{p_j-r+\alpha-1+\eta-\beta}(f_j)\otimes_{r-\alpha+1} g\|^2  \|D^{\alpha +\eta} [\partial_j\varphi(F)] \|^2 \right] \right)^{\frac 12} \\
&\le& \left(    E\|I_{p_j-r+\alpha-1+\eta-\beta}(f_j)\otimes_{r-\alpha+1} g\|^4 \right)  ^{\frac 14}\left( E \|D^{\alpha +\eta} [\partial_j\varphi(F)] \|^4 \right)  ^{\frac 14}.
\end{eqnarray*}
Set $\gamma =\alpha + \eta$. Applying the generalized Fa\'a di Bruno's formula (see, e.g., \cite{Mis}) we deduce that
\[
D^\gamma [\partial_j\varphi(F) ] = \sum   \frac {\gamma !} { \prod_{i=1}^\gamma i!^{k_i}  \prod_{i=1}^{\gamma} \prod_{j=1}^m q_{ij}!}   \frac{\partial^{k}\partial_j\varphi(F)}{\partial x^{p_1}_1\cdots \partial x_m^{p_m}}   \prod_{i=1}^{\gamma \wedge p^*} (D^i F_1) ^{\otimes q_{i1}} \otimes \cdots \otimes (D^i F_m) ^{\otimes q_{im}},
\]
where $p^* =  \min\{p_1,...,p_m\}$, and the sum runs over all nonnegative integer solutions of the system of $\gamma+1$ equations
\begin{eqnarray*}
&&k_1+2k_2+\cdots +\gamma k_\gamma = \gamma, \\
&&q_{11} + q_{12}+\cdots +q_{1m} = k_1,\\
&&q_{21} + q_{22}+\cdots +q_{2m} = k_2,\\
&&\cdots\cdots \\
&&q_{\gamma1} + q_{\gamma2}+\cdots +q_{\gamma m} = k_\gamma,
\end{eqnarray*}
and we have moreover set $p_j = q_{1j}+\cdots +q_{\gamma j}$, $j=1,...,r$, and
$k = p_1+\cdots +p_m = k_1+\cdots +k_\gamma$.  This expression yields immediately that
\[
\|D^\gamma [\partial_j\varphi(F) ]\| \le c \|\varphi\|_q \sum  \prod_{i=1}^{\gamma \wedge p^*}
\| D^i F_1 \|^{ q_{i1}}  \cdots  \|D^i F_m\|^{ q_{im}}
\]
and using the facts that all $\mathbb{D}^{k,p}$ norms ($k,p\geq 1$) are equivalent on a fixed Wiener chaos and that the elements of the vector $F$ have unit variance by assumption, we infer that
\[
\left(E \|D^{\gamma} [\partial_j\varphi(F)] \|^4 \right)  ^{\frac 14} \le c\|\varphi\|_q.
\]
On the other hand, using hypercontractivity one has that
\[
\left( E\|I_{p_j-r+\alpha-1+\eta-\beta}(f_j)\otimes_{r-\alpha+1} g\|^4 \right) ^{\frac 14}
\le
c\left( E\|I_{p_j-r+\alpha-1+\eta-\beta}(f_j)\otimes_{r-\alpha+1} g\|^2 \right) ^{\frac 12}.
\]
Since
\[
  E\|I_{p_j-r+\alpha-1+\eta- \beta}(f_j)\otimes_{r-\alpha+1} g\|^2
 =(p_j-r+\alpha-1+\eta- \beta)! \| f_j\otimes_{r-\alpha+1} g \|^2,
 \]
 and
 \[
 \max_{1\le r \le q} \| f_j \otimes_{r} g \|\leq {\rm Cov}(F_j^2,G^2)\quad\mbox{(see, e.g., \cite[inequality (3.26)]{NouRos})},
 \]
 we finally obtain
 \[
 E| \langle (I-L)^{-1} \partial_j\varphi(F) DF_j , DG \rangle_\HH|\le c
  \|\varphi\|_q {\rm Cov}(F_j^2,G^2),
 \]
thus concluding the proof.
\end{proof}

\medskip

\noindent{\bf Proof of Theorem \ref{t:main}}. Just observe that

\begin{eqnarray*}
 && \left| E\left[\prod_{j=1}^d \psi_j(F_j)\right] - \prod_{j=1}^d E[\psi_j(F_j)]\right|\\
 &\leq& \sum_{j=2}^d \Big| E[\psi_1(F_1)\cdots \psi_{j-1}(F_{j-1})]\,  E[\psi_j(F_j)]\,\cdots\, E[\psi_d(F_d)]  \\
 &&\quad\quad\quad\quad \quad- E[\psi_1(F_1)\cdots \psi_{j}(F_{j})]\, E[\psi_{j+1}(F_{j+1})]\,\cdots\, E[\psi_d(F_d )]\Big| ,
\end{eqnarray*}
so that the conclusion is achieved (after some routine computations) by applying Proposition \ref{p:main}
(in the case $m=j-1$, $p_i = q_i$, $i=1,...,j-1$, and $q = q_j$) to each summand on the right-hand side
of the previous estimate. \qed

\medskip

\subsection{Proof of Theorem \ref{t:nr+}}

 The equivalence between (1) and (2) follows from \cite[Theorem 3.4]{NouRos}.
That (3) implies (1) would have been immediate if the square function $x\mapsto x^2$ were bounded.
To overcome this slight difficulty, it suffices to combine the hypercontractivity property
of chaotic random variables (from which it follows that our sequence $(F_n)$ is bounded in $L^p(\Omega)$ for any $p\geq 1$)
with a standard approximation argument.
Finally, the implication $(1)\Rightarrow(3)$ is a direct consequence of (\ref{e:mainresult}).

\subsection{Proof of Theorem \ref{utile}}

Assume that there exists a subsequence of $\{F_n\}$ converging in distribution to some limit $(V_1,\ldots,V_d)$.
For any collection of smooth test functions $\psi_1,...,\psi_d:\R\to\R$, one can then write
\begin{equation}\label{eq112}
E\left[ \prod_{j=1}^d\psi_j(V_j)\right]
=
\prod_{j=1}^dE\left[ \psi_j(V_j)\right]
=
\prod_{j=1}^dE\left[ \psi_j(U_j)\right]
=
E\left[ \prod_{j=1}^d\psi_j(U_j)\right].
\end{equation}
Indeed, the first equality in (\ref{eq112}) is a direct consequence of Theorem \ref{t:nr+}, the second one follows from the fact that $V_j\overset{\rm law}{=}U_j$ for any $j$ by assumption, and the last one follows from
the independence of the $U_j$.
Thus, we deduce from (\ref{eq112}) that $(U_1,\ldots,U_d)$ is the only possible limit in law
for any converging subsequence extracted from $\{F_n\}$.
Since the sequence $\{F_n\}$ is tight (indeed, it is bounded in $L^{2}(\Omega)$), one deduces that $F_n\overset{\rm law}{\to}(U_1,\ldots,U_d)$,
which completes the proof of Theorem \ref{utile}.

\end{document}